\documentclass[11pt,a4paper]{amsart}

\usepackage{amsmath}
\usepackage{amsfonts}
\usepackage{amssymb}
\usepackage{amsthm}
\usepackage{epsfig}
\usepackage{graphicx}
\usepackage{url}
\usepackage[usenames,dvipsnames]{xcolor}
\usepackage[colorlinks=true,linkcolor=Blue,citecolor=Green]{hyperref}
\usepackage{nicefrac}
\usepackage{aliascnt}		
\usepackage[alwaysadjust]{paralist}

\newtheorem{theorem}{Theorem}[section]
\newaliascnt{lemma}{theorem}
\newaliascnt{corollary}{theorem}
\newaliascnt{definition}{theorem}
\newaliascnt{remark}{theorem}
\newaliascnt{proposition}{theorem}
\newaliascnt{conjecture}{theorem}
\newaliascnt{example}{theorem}
\newaliascnt{problem}{theorem}

\aliascntresetthe{lemma}

\newtheorem*{lemma*}{Lemma}

\newtheorem{corollary}[corollary]{Corollary}
\aliascntresetthe{corollary}

\newtheorem*{corollary*}{Corollary}

\newtheorem{definition}[definition]{Definition}
\aliascntresetthe{definition}

\newtheorem*{definition*}{Definition}

\newtheorem{remark}[remark]{Remark}
\aliascntresetthe{remark}

\newtheorem*{remark*}{Remark}

\newtheorem{proposition}[proposition]{Proposition}
\aliascntresetthe{proposition}

\newtheorem*{proposition*}{Proposition}

\newtheorem{conjecture}[conjecture]{Conjecture}
\aliascntresetthe{conjecture}

\newtheorem*{conjecture*}{Conjecture}

\aliascntresetthe{example}

\newtheorem*{example*}{Example}

\newtheorem{problem}[problem]{Problem}
\aliascntresetthe{problem}

\newtheorem*{problem*}{Problem}

\DeclareMathOperator{\conv}{conv}

\def\P{\mathcal{P}}

\def\C{\mathcal{C}}
\def\R{\mathbb{R}}
\def\C{\mathbb{C}}
\def\Z{\mathbb{Z}}
\def\N{\mathbb{N}}
\def\i{\mathrm{i}}

\DeclareMathOperator{\vol}{vol}

\DeclareMathOperator{\aff}{aff}

\DeclareMathOperator{\LE}{G}
\DeclareMathOperator{\lE}{g}

\numberwithin{equation}{section}

\begin{document}

\title[On counterexamples to a conjecture of Wills]{On counterexamples to a conjecture of Wills and Ehrhart polynomials whose roots have equal real parts}
\author{Matthias Henze}
\address{Institut f\"ur Informatik, Freie Universit\"at Berlin, Takustra\ss e 9, 14195 Berlin, Germany}
\email{matthias.henze@fu-berlin.de}

\thanks{The work was partially supported by the Deutsche Forschungsgemeinschaft (DFG) within the project He 2272/4-1 and by
the ESF EUROCORES programme EuroGIGA-VORONOI, (DFG): Ro 2338/5-1.}



\begin{abstract}
As a discrete analog to Minkowski's theorem on convex bodies, Wills conjectured that the Ehrhart coefficients of a centrally symmetric lattice polytope
with exactly one interior lattice point are maximized by those of the cube of side length two.
We discuss several counterexamples to this conjecture and, on the positive side, we identify a family of lattice polytopes that fulfill the claimed inequalities.
This family is related to the recently introduced class of $l$-reflexive polytopes.
\end{abstract}

\maketitle

\section{Introduction}

Let $\P^n$ denote the family of all \emph{lattice polytopes} in $\R^n$, that is, the convex hulls of finitely many points from the integer lattice $\Z^n$.
Such a polytope $P\in\P^n$ is called \emph{centrally symmetric} if $P=-P$.
Its \emph{volume} (Lebesgue measure) is denoted by $\vol(P)$ and its discrete volume, the \emph{lattice point enumerator}, by $\LE(P)=\#(P\cap\Z^n)$.

Minkowski~\cite{minkowski1896geometrie} proved that among the centrally symmetric compact convex sets that have only the origin as an interior lattice point,
the cube $C_n=[-1,1]^n$ has the biggest volume and contains the most lattice points.
More precisely, for every such set $K\subset\R^n$,
\begin{align}
\vol(K)\leq\vol(C_n)=2^n\qquad\textrm{and}\qquad\LE(K)\leq\LE(C_n)=3^n.\label{eqn_minkowski}
\end{align}

Wills~\cite{wills1981onan} proposed to further discretize these inequalities for the class of lattice polytopes.
A famous result of Ehrhart~\cite{ehrhart1968sur} shows that the counting function $\N\to\N,k\mapsto\LE(kP)$, is a polynomial in $k$
of degree~$n$, whenever $P\in\P^n$ is a lattice polytope.
This polynomial $\LE(kP)=\sum_{i=0}^n\lE_i(P)k^i$ is called the \emph{Ehrhart polynomial} of $P$ and the numbers $\lE_i(P)$ are the \emph{Ehrhart coefficients}.
It is not hard to see that $\lE_0(P)=1$ and $\lE_n(P)=\vol(P)$.
Moreover, also the second highest Ehrhart coefficient~$\lE_{n-1}$ has a nice geometric meaning which is given in detail in~\eqref{eqn_coeff_n-1}.
For details on Ehrhart theory including extensive references we refer the interested reader to~\cite{beckrobins2007computing}.
Now, using the description for $\lE_{n-1}(P)$, Wills proved that for every centrally symmetric $P\in\P^n$ that has only the origin as an interior lattice point
\[\lE_{n-1}(P)\leq\lE_{n-1}(C_n)=n2^{n-1}.\]
Furthermore, he wondered about a much stronger extremality property of the cube~$C_n$.

\begin{conjecture}[Wills' Conjecture~\cite{wills1981onan,gritzmannwills1993lattice}]\label{conj_ch4_wills_gi}
Let $P\in\P^n$ be a centrally symmetric lattice polytope with the origin being its only interior lattice point. Then
\[\lE_i(P)\leq\lE_i(C_n)=2^i\binom{n}{i}\quad\textrm{for all}\quad i=0,\ldots,n.\]
\end{conjecture}

Minkowski's inequalities~\eqref{eqn_minkowski} nicely embed into these proposed relations, since $\vol(P)=\lE_n(P)$ and $\LE(P)=\sum_{i=0}^n\lE_i(P)$.
Wills proved his conjecture for $n=3$, and it gets further support by the fact that the Ehrhart polynomial of every $P\in\P^n$, that meets the conditions above,
is pointwise maximized by that one of the cube. That is, $\LE(kP)\leq\LE(kC_n)$, for $k\in\N$, which follows from a discrete version of Minkowski's $1$st Theorem due to
Betke, Henk \& Wills~\cite[Thm.~2.1]{betkehenkwills1993successive}.

Nevertheless, the objective of this paper is to exhibit counterexamples to Wills' Conjecture.
On the positive side, we identify a family of lattice polytopes that fulfill the claimed inequalities (\autoref{cor_gen_ineq}).
This family is related to the recently introduced $l$-reflexive polytopes which we also briefly discuss.

\section{Counterexamples to Wills' conjecture}

Let $C_n^\star=\conv\{\pm e_1,\ldots,\pm e_n\}$ be the standard crosspolytope in $\P^n$.
Consider the polytope $P_n=\conv\{C_{n-1}\times\{0\},C_{n-1}^\star\times\{-1,1\}\}$
that arises as the convex hull of an $(n-1)$-dimensional cube put at height~$0$ and two $(n-1)$-dimensional crosspolytopes put at height $-1$ and $1$, respectively.
Using \texttt{polymake}~\cite{gawrilowjoswig2000polymake} and \texttt{LattE}~\cite{latte2011ausers}, we see that $P_7$ is a counterexample to \autoref{conj_ch4_wills_gi}.
More precisely, its Ehrhart polynomial is given by
\[1+\frac{1534}{105}k+\frac{3188}{45}k^2+\frac{7112}{45}k^3+\frac{1756}{9}k^4+\frac{7004}{45}k^5+\frac{4952}{45}k^6+\frac{15656}{315}k^7,\]
and therefore $\lE_1(P_7)=\frac{1534}{105}>14=\lE_1(C_7)$.

The following relations on Ehrhart coefficients of products of lattice polytopes help us to construct counterexamples in every dimension $n\geq7$.
For the sake of completeness, we provide the short argument (see also~\cite[Exs.~2.4]{beckrobins2007computing}).

\begin{proposition}
Let $P\in\P^p$ and $Q\in\P^q$ be lattice polytopes. Then
\[\lE_j(P\times Q)=\sum_{i=0}^j\lE_i(P)\lE_{j-i}(Q)\quad\textrm{for all}\quad j=0,\ldots,p+q,\]
where $\lE_k(P)=\lE_l(Q)=0$, for all $k>p$ and $l>q$. 
\end{proposition}
\begin{proof}
For every $k\in\N$, we have
\begin{align*}
\LE(k(P\times Q))&=\LE(kP\times kQ)=\LE(kP)\LE(kQ)\\
&=\left(\sum_{i=0}^p\lE_i(P)k^i\right)\left(\sum_{j=0}^q\lE_j(Q)k^j\right)\\
&=\sum_{j=0}^{p+q}\left(\sum_{i=0}^j\lE_i(P)\lE_{j-i}(Q)\right)k^j.
\end{align*}
Comparing coefficients gives the claimed identities.
\end{proof}

The identity in the case $j=1$ implies
\[\lE_1(P_7 \times C_m)=\lE_1(P_7)+\lE_1(C_m)>\lE_1(C_7)+\lE_1(C_m)=\lE_1(C_{m+7}),\]
for all $m\in\N_0$.
Hence, Wills' Conjecture fails in every dimension $n\geq7$.

Another set of lattice polytopes that are of even simpler structure than the~$P_n$'s exhibit counterexamples also for higher Ehrhart coefficients.
Consider the polytope $Q_n=\conv\{C_{n-1}\times\{0\},\pm e_n\}$, i.e., a bipyramid over an $(n-1)$-dimensional cube (already appeared in~\cite[Prop.~1.1]{beyhenkhenzelinke2011notes}).
Invoking \texttt{polymake} and \texttt{LattE} again, we find
\begin{align*}
\lE_1(Q_9)=\frac{494}{15}&>18=\lE_1(C_9),\\
\lE_3(Q_{11})=1976&>1320=\lE_3(C_{11})\quad\textrm{and}\\
\lE_5(Q_{13})=\frac{260832}{5}&>41184=\lE_5(C_{13}).
\end{align*}

Moreover, the first Ehrhart coefficient of the bipyramid $Q_n$ behaves quite wildly and in dimensions $n=4k+1$ the proposed bound $\lE_1(P)\leq 2n$ in \autoref{conj_ch4_wills_gi} fails badly.
Recall that the Landau notation $g(n)\in\Theta(f(n))$ means that there are constants $c_1,c_2>0$ such that $c_1f(n)\leq g(n)\leq c_2f(n)$, for all large enough $n\in\N$.

\begin{proposition}\label{prop_ch4_counterexs_coeff}
For $n\in\N$, let $Q_n=\conv\left\{C_{n-1}\times\{0\},\pm e_n\right\}$. Then
\[\lE_1(Q_n)=2(n-1)+(4-2^n)B_{n-1},\]
where $B_j$ denotes the $j$th Bernoulli number.

In particular, we have
\[\lE_1(Q_n)=2(n-1)\quad\textrm{for even }n\in\N,\]
and
\[(-1)^{\frac{n-1}2}\lE_1(Q_n)\in\Theta\left(\left(\frac{n}{\pi e}\right)^n\right)\quad\textrm{for odd }n\in\N.\]
\end{proposition}
\begin{proof}
Let $k\in\N$ and let $H_j=\{x\in\R^n:x_n=j\}$ be the orthogonal plane to $e_n$ of height $j$.
For $j=0,1,\ldots,k$, the plane $H_j$ intersects $kQ_n$ in a copy of $C_{n-1}$ scaled by $k-j$.
Counting the lattice points in $kQ_n$ with respect to these intersections, we find that the Ehrhart polynomial of $Q_n$ is given by
\begin{align*}
\LE(kQ_n)&=(2k+1)^{n-1}+2\sum_{j=0}^{k-1}(2j+1)^{n-1}\\
&=(2k+1)^{n-1}+2\sum_{j=0}^{k-1}\sum_{i=0}^{n-1}\binom{n-1}{i}(2j)^i\\
&=\sum_{i=0}^{n-1}\binom{n-1}{i}2^ik^i+2\sum_{i=0}^{n-1}\binom{n-1}{i}2^i\left(\sum_{j=0}^{k-1}j^i\right).
\end{align*}
Faulhaber's formula (see~\cite[\textsection23.1]{abramstegun1992handbook}) expresses the sum $\sum_{j=0}^{k-1}j^i$ as a polynomial in $k$, more precisely
\[\sum_{j=0}^{k-1}j^i=\frac1{i+1}\sum_{j=0}^i\binom{i+1}{j}B_jk^{i-j+1}=\frac1{i+1}\sum_{j=1}^{i+1}\binom{i+1}{j}B_{i-j+1}k^j,\]
where $B_j$ is the $j$th Bernoulli number\index{Bernoulli number}. Therefore, we continue as
\begin{align*}
\LE(kQ_n)&=\sum_{i=0}^{n-1}\binom{n-1}{i}2^ik^i+2\sum_{i=0}^{n-1}\binom{n-1}{i}\frac{2^i}{i+1}\sum_{j=1}^{i+1}\binom{i+1}{j}B_{i-j+1}k^j\\
&=\sum_{i=0}^{n-1}\binom{n-1}{i}2^ik^i+\frac2{n}\sum_{j=1}^n\sum_{i=j-1}^{n-1}\binom{n}{i+1}2^i\binom{i+1}{j}B_{i-j+1}k^j.
\end{align*}
Thus, the Ehrhart coefficients of $Q_n$ are given by
\begin{align*}
\lE_i(Q_n)=\binom{n-1}{i}2^i+\frac2{n}\sum_{j=i-1}^{n-1}\binom{n}{j+1}2^j\binom{j+1}{i}B_{j-i+1},
\end{align*}
for all $i=1,\ldots,n$. In the case $i=1$, this gives us
\begin{align*}
\lE_1(Q_n)&=2(n-1)+2\sum_{j=0}^{n-1}\binom{n-1}{j}2^jB_j=2(n-1)+2^nB_{n-1}\left(\frac12\right),
\end{align*}
where $B_n(x)=\sum_{j=0}^n\binom{n}{j}x^{n-j}B_j$ denotes the $n$th Bernoulli polynomial.
In~\cite[\textsection23.1]{abramstegun1992handbook}, we find the identity $B_n\left(\frac12\right)=-(1-2^{1-n})B_n$, which leads to our desired expression $\lE_1(Q_n)=2(n-1)+(4-2^n)B_{n-1}$.

Since $B_j=0$ for all odd indices $j\geq3$, we have $\lE_1(Q_n)=2(n-1)$ for all even $n\in\N$.
The Bernoulli numbers with even indices satisfy (see~\cite[\textsection23.1]{abramstegun1992handbook})
\[\frac{2(2j)!}{(2\pi)^{2j}}<(-1)^{j+1}B_{2j}<\frac{2(2j)!}{(2\pi)^{2j}}\left(\frac1{1-2^{1-2j}}\right).\]
Assuming that $n=2k+1$ for some $k\in\N$, we obtain
\begin{align*}
(-1)^k\lE_1(Q_{2k+1})&=(-1)^k4k+(-1)^k(4-2^{2k+1})B_{2k}\\
&\in\Theta\left(\frac{2^{2k+2}(2k)!}{(2\pi)^{2k}}\right)=\Theta\left(\left(\frac{n}{\pi e}\right)^n\right).
\end{align*}
The last equality comes from Stirling's approximation of the factorial.
\end{proof}

\section{Lattice polytopes with Ehrhart polynomials with roots of equal real part}\label{sect_roots}

This section deals with linear inequalities in the spirit of Wills' Conjecture for a special class of lattice polytopes.
The Ehrhart polynomial of a lattice polytope $P\in\P^n$ may be understood as a polynomial in a complex variable and thus it makes sense to speak about its roots.
Following \cite{beyhenkwills2007notes}, we denote these roots by $-\gamma_1(P),\ldots,-\gamma_n(P)$.
Using $\lE_n(P)=\vol(P)$, we get
\[\LE(sP)=\sum_{i=0}^n\lE_i(P)s^i=\vol(P)\prod_{i=1}^n(s+\gamma_i(P))\quad\textrm{for}\quad s\in\C.\]

Braun~\cite{braun2008norm} proved that the roots of an Ehrhart polynomial lie in the disc with center $-\frac12$ and radius $n(n-\frac12)$.
This is only one reason to study the situation in which the real part of all roots equals $-\frac12$ as it was done for example in~\cite{beckrobins2007computing,beyhenkwills2007notes}.
The cube $C_n$ and the crosspolytope $C_n^\star$ are standard examples of lattice polytopes that belong to this class (see the references above).
Our main result here is concerned with a broader class of lattice polytopes.

\begin{theorem}\label{thm_ch4_wills_conj_-12}
Let $P\in\P^n$ be a lattice polytope with the property that all the roots of its Ehrhart polynomial have real part $-\frac1{a}$, for some $a>0$.
\begin{enumerate}[i)]
 \item For all $0\leq s< t\leq n$, we have
  \[\frac{\lE_t(P)}{\lE_s(P)}\leq a^{t-s}\frac{\binom{n}{t}}{\binom{n}{s}}.\]
  For $(s,t)=(n-1,n)$ we have equality. For every $0\leq s<t\leq n$ with $(s,t)\neq(n-1,n)$ equality holds if and only if $\LE(kP)=(ak+1)^n$, for all $k\in\N$.
 \item We have \[\vol(P)\leq\left(\frac{a}{a+1}\right)^n\LE(P),\]
  and equality holds if and only if $\LE(kP)=(ak+1)^n$, for all $k\in\N$.
 \item We have \[\LE(P)\leq\frac{(a+1)^{n-2}(a+2)}{a^{n-1}}\vol(P)+(a+1)^{n-2}.\]
  Equality holds if and only if there is at most one pair of complex conjugate roots with nonzero imaginary part. In particular, equality holds for $n\in\{2,3\}$.
\end{enumerate}
\end{theorem}
\begin{proof}
\romannumeral1): As above, we write
\[\LE(kP)=\vol(P)\prod_{i=1}^n\left(k+\gamma_i(P)\right).\]
By assumption, the real part of $-\gamma_i(P)$ equals $-\frac1{a}$, for all $i=1,\ldots,n$.

We consider the case $n=2l$ first. There are $b_1,\ldots,b_l\in\R$ such that
\begin{align}
\frac{\LE(kP)}{\vol(P)}&=\prod_{j=1}^l\left(k+\frac1{a}\pm b_j\i\right)=\prod_{j=1}^l\left(\left(k+\frac1{a}\right)^2+b_j^2\right)\label{eqn_main_thm_1}\\
&=\sum_{j=0}^l\sigma_{l-j}\left(b_1^2,\ldots,b_l^2\right)\left(k+\frac1{a}\right)^{2j}\nonumber\\
&=\sum_{j=0}^l\sigma_{l-j}\left(b_1^2,\ldots,b_l^2\right)\sum_{t=0}^{2j}\binom{2j}{t}\left(\frac1{a}\right)^{2j-t}k^t\nonumber\\
&=\sum_{t=0}^{2l}a^t\left(\sum_{j=\lceil\frac{t}2\rceil}^la^{-2j}\sigma_{l-j}\left(b_1^2,\ldots,b_l^2\right)\binom{2j}{t}\right)k^t.\nonumber
\end{align}
As usual, $\sigma_j(x_1,\ldots,x_l)=\sum_{1\leq i_1<\ldots<i_j\leq l}\prod_{t=1}^j x_{i_t}$ denotes the $j$th elementary symmetric polynomial.
From the above identity we can read off a formula for the Ehrhart coefficients $\lE_t(P)$ and
it follows that, for $0\leq s\leq t\leq n$, the inequality
\[\frac{\lE_t(P)}{\lE_s(P)}=a^{t-s}\frac{\sum_{j=\lceil\frac{t}2\rceil}^la^{-2j}\sigma_{l-j}\left(b_1^2,\ldots,b_l^2\right)\binom{2j}{t}}
{\sum_{j=\lceil\frac{s}2\rceil}^la^{-2j}\sigma_{l-j}\left(b_1^2,\ldots,b_l^2\right)\binom{2j}{s}}\leq a^{t-s}\frac{\binom{n}{t}}{\binom{n}{s}}\]
is equivalent to
\begin{align}
&\sum_{j=\lceil\frac{t}2\rceil}^la^{-2j}\sigma_{l-j}\left(b_1^2,\ldots,b_l^2\right)\binom{2j}{t}\binom{2l}{s}\nonumber\\
&\qquad\qquad\qquad\leq\sum_{j=\lceil\frac{s}2\rceil}^la^{-2j}\sigma_{l-j}\left(b_1^2,\ldots,b_l^2\right)\binom{2j}{s}\binom{2l}{t}.\label{eqn_main_thm_2}
\end{align}
Since all summands are nonnegative and $t\geq s$, it suffices to show that $\binom{2j}{t}\binom{2l}{s}\leq\binom{2j}{s}\binom{2l}{t}$, for all $j=\lceil\frac{t}2\rceil,\ldots,l$.
As this is equivalent to $(2j-t+1)\cdot\ldots\cdot(2j-s)\leq(2l-t+1)\cdot\ldots\cdot(2l-s)$, we are done.

Since $t\geq s$, we have the same number of summands on either side of \eqref{eqn_main_thm_2} if and only if $t$ is even and $s=t-1$.
Moreover, since $s\neq t$, we have $\binom{2j}{t}\binom{2l}{s}\leq\binom{2j}{s}\binom{2l}{t}$ if and only if $j=l$.
These two observations imply that we have equality in~\eqref{eqn_main_thm_2} for $(s,t)=(2l-1,2l)=(n-1,n)$, and for every other pair $(s,t)$ with $0\leq s<t\leq n$
if and only if $b_1=\ldots=b_l=0$, which is equivalent to $\LE(kP)=(ak+1)^n$, for all $k\in\N$.
Note, that it is not clear that this means that $P$ is unimodularly equivalent to $\frac{a}2C_n$ (for related work see~\cite{haasemcallister2008quasi}).

The case of odd dimensions $n=2l+1$ is similar. In comparison to the even-dimensional case, there is an additional real zero $-\frac1{a}$ and we get
\begin{align*}
\frac{\LE(kP)}{\vol(P)}&=\left(k+\frac1{a}\right)\prod_{j=1}^l\left(k+\frac1{a}\pm b_j\i\right)\\
&=\left(k+\frac1{a}\right)\sum_{t=0}^{2l}a^t\left(\sum_{j=\lceil\frac{t}2\rceil}^la^{-2j}\sigma_{l-j}\left(b_1^2,\ldots,b_l^2\right)\binom{2j}{t}\right)k^t\\
&=\sum_{t=1}^{2l+1}a^{t-1}\left(\sum_{j=\lceil\frac{t-1}2\rceil}^la^{-2j}\sigma_{l-j}\left(b_1^2,\ldots,b_l^2\right)\binom{2j}{t-1}\right)k^t\\
&\phantom{=}\ +\sum_{t=0}^{2l}a^{t-1}\left(\sum_{j=\lceil\frac{t}2\rceil}^la^{-2j}\sigma_{l-j}\left(b_1^2,\ldots,b_l^2\right)\binom{2j}{t}\right)k^t\\
&=\sum_{t=0}^{2l+1}a^{t-1}\left(\sum_{j=\lceil\frac{t-1}2\rceil}^la^{-2j}\sigma_{l-j}\left(b_1^2,\ldots,b_l^2\right)\binom{2j+1}{t}\right)k^t.
\end{align*}
Therefore, the desired inequality $\frac{\lE_t(P)}{\lE_s(P)}\leq a^{t-s}\frac{\binom{n}{t}}{\binom{n}{s}}$ is equivalent to
\begin{align*}
&\sum_{j=\lceil\frac{t-1}2\rceil}^la^{-2j}\sigma_{l-j}\left(b_1^2,\ldots,b_l^2\right)\binom{2j+1}{t}\binom{2l+1}{s}\\
&\qquad\qquad\qquad\leq\sum_{j=\lceil\frac{s-1}2\rceil}^la^{-2j}\sigma_{l-j}\left(b_1^2,\ldots,b_l^2\right)\binom{2j+1}{s}\binom{2l+1}{t},
\end{align*}
and the further analysis is analogous to the even-dimensional case.

\romannumeral2): This follows directly from part \romannumeral1), since
\[a^n\LE(P)=\sum_{s=0}^na^n\lE_s(P)\geq\sum_{s=0}^n\binom{n}{s}a^s\lE_n(P)=(a+1)^n\vol(P).\]
The characterization of equality is inherited from part \romannumeral1)\,as well.

\romannumeral3): Multiplying either side of the claimed inequality by $\frac{a^n}{\vol(P)}$, shows that we need to prove
\[\frac{a^n\LE(P)}{\vol(P)}\leq(a+1)^n-(a+1)^{n-2}+(a+1)^{n-2}\frac{a^n}{\vol(P)}.\]
Again, we consider the case $n=2l$ first. Using Equation~\eqref{eqn_main_thm_1} for $k=0$ and $k=1$ gives the equivalent inequality
\[\prod_{j=1}^l\left((a+1)^2+(ab_j)^2\right)\leq(a+1)^{2l}-(a+1)^{2l-2}+(a+1)^{2l-2}\prod_{j=1}^l\left(1+(ab_j)^2\right).\]
Expanding the products yields that this is equivalent to
\[\sum_{j=0}^{l-1}(a+1)^{2j}\sigma_{l-j}\left((ab_1)^2,\ldots,(ab_l)^2\right)\leq\sum_{j=0}^{l-1}(a+1)^{2l-2}\sigma_{l-j}\left((ab_1)^2,\ldots,(ab_l)^2\right),\]
which even holds summand-wise.

Equality holds if and only if $\sigma_{l-j}\left((ab_1)^2,\ldots,(ab_l)^2\right)=0$, for all $j=0,\ldots,l-2$.
Since $a>0$, this holds if and only if at most one of $b_1,\ldots,b_l$ is nonzero.
This means, that there is at most one conjugate pair of roots of $\LE(kP)$ having nonzero imaginary part.

The case $n=2l+1$ is completely analogous.
\end{proof}

\begin{remark}\label{rem_gen_ineq}
The inequality in \autoref{thm_ch4_wills_conj_-12}~\romannumeral3) generalizes the linear inequality that is part of the characterization of $4$-dimensional lattice polytopes
having only roots with real part equal to $-\frac12$ (see~\cite[Prop.~1.9]{beyhenkwills2007notes}).
\end{remark}

Applying \autoref{thm_ch4_wills_conj_-12}~\romannumeral1) with $a=2$ and $s=0$ gives the following positive result concerning Wills' Conjecture.

\begin{corollary}\label{cor_gen_ineq}
\autoref{conj_ch4_wills_gi} holds for every lattice polytope with the property that all the roots of its Ehrhart polynomial have real part equal to $-\frac12$.
\end{corollary}

One may wonder whether lattice polytopes with an Ehrhart polynomial all of whose roots have real part equal to $-\frac1{a}$ actually exist.
Let $P$ be a lattice polytope such that all roots of its Ehrhart polynomial have real part~$-\frac12$.
Such polytopes are quite abundant, see~\cite{beyhenkwills2007notes}.
Now, for every $l\in\N$, the roots of the Ehrhart polynomial of $lP$ have real part~$-\frac1{2l}$.
This construction is not very satisfactory though.

Kasprzyk \& Nill~\cite{kasprzyknill2012reflexive} introduced the class of \emph{$l$-reflexive polytopes} which contains more interesting examples with the desired property.
Before we can give their definition, we need to fix some notation.
A lattice point $u\in\Z^n\setminus\{0\}$ is \emph{primitive} if the only lattice points contained in the line segment from $0$ to $u$ are its endpoints.
Every lattice polytope $P\in\P^n$ has a unique irredundant representation $P=\{x\in\R^n:u_i^\intercal x\leq l_i, i=1,\ldots,m\}$, where $u_i\in\Z^n$ are nonzero primitive lattice vectors
and the $l_i$ are natural numbers.
The \emph{index} of $P$ is defined as the least common multiple of $l_1,\ldots,l_m$.

\begin{definition}[$l$-reflexive polytope]\label{def_l-reflexive}
A lattice polytope $P\in\P^n$ is called \emph{$l$-reflexive}, for some $l\in\N$, if
\begin{enumerate}[i)]
 \item the origin is contained in the interior of $P$,
 \item the vertices of $P$ are primitive,
 \item $l_i=l$, for all $i=1,\ldots,m$.
\end{enumerate}
\end{definition}

Note, that $1$-reflexive polytopes are the much studied \emph{reflexive polytopes} which have a close connection to algebraic geometry
(see~\cite{nill2005diss} and the references therein).

Now, Kasprzyk \& Nill study properties of $l$-reflexive polytopes and give a classification algorithm for the planar case.
For example, they find~$3605$ $l$-reflexive polygons with an index at most $60$.
The relation to our question is~\cite[Prop.~17]{kasprzyknill2012reflexive} which says that every $l$-reflexive polygon, not unimodularly equivalent to the triangle
$\conv\{(-1,-1),(-1,2),(2,-1)\}$, has an Ehrhart polynomial all of whose roots have real part equal to $-\frac1{2l}$.

We conclude by extending characterizations given in~\cite{kasprzyknill2012reflexive,beyhenkwills2007notes}.
Recall that the \emph{polar} of a polytope $P\in\P^n$ is defined as $P^\star=\{x\in\R^n:x^\intercal y\leq 1,\textrm{ for all }y\in P\}$.

\begin{proposition}\label{prop_charact_l-refl}
Let $P\in\P^n$ be a lattice polytope with index $l$ that contains the origin in its interior.
Then, the following are equivalent:
\begin{enumerate}[i)]
 \item $P$ is an $l$-reflexive polytope,
 \item $lP^\star$ is a lattice polytope,
 \item $\lE_{n-1}(P)=\frac{n}{2l}\vol(P)$.
\end{enumerate}
\end{proposition}
\begin{proof}
The equivalence of \romannumeral1) and \romannumeral2) is shown in \cite[Prop.~2]{kasprzyknill2012reflexive}.

\romannumeral1) $\Longleftrightarrow$ \romannumeral3): Let $F_i=\{x\in\R^n:u_i^\intercal x=l_i\}\cap P$, for $i=1,\ldots,m$, be the facets of $P$, where as before the $u_i$ are primitive normal vectors.
As Ehrhart~\cite{ehrhart1967sur} already showed (cf.~\cite[Thm.~5.6]{beckrobins2007computing})
\begin{align}
\lE_{n-1}(P)&=\frac12\sum_{i=1}^m\frac{\vol_{n-1}(F_i)}{\det(\aff F_i\cap\Z^n)}.\label{eqn_coeff_n-1}
\end{align}
Therein, $\vol_{n-1}(F_i)$ denotes the $(n-1)$-dimensional volume of the facet $F_i$, $\aff F_i$ its affine hull, and $\det(\aff F_i\cap\Z^n)=\|u_i\|$
(see~\cite[Prop.~1.2.9]{martinet2003perfect}) the determinant of the sublattice of $\Z^n$ contained in $\aff F_i$.

Now, if $P$ is $l$-reflexive, then $l_i=l$ for all $i=1,\ldots,m$, and writing $F_i^o=\conv\{0,F_i\}$, we have
\begin{align*}
\lE_{n-1}(P)&=\frac12\sum_{i=1}^m\frac{\vol_{n-1}(F_i)}{\det(\aff F\cap\Z^n)}=\frac{n}{2l}\sum_{i=1}^m\frac{l\cdot\vol_{n-1}(F_i)}{n\cdot||u_i||}\\
&=\frac{n}{2l}\sum_{i=1}^m\vol(F_i^o)=\frac{n}{2l}\vol(P).
\end{align*}
Conversely, let $\lE_{n-1}(P)=\frac{n}{2l}\vol(P)$. Similarly as above, we get
\[\sum_{i=1}^m\frac{l_i\cdot\vol_{n-1}(F_i)}{n\cdot||u_i||}=\vol(P)=\frac{2l}{n}\lE_{n-1}(P)=\sum_{i=1}^m\frac{l\cdot\vol_{n-1}(F_i)}{n\cdot||u_i||}.\]
Therefore, $\sum_{i=1}^m(l-l_i)\frac{\vol_{n-1}(F_i)}{||u_i||}=0$.
Since $l\geq l_i$, for all $i=1,\ldots,m$, we get $l=l_i$, for all $i=1,\ldots,m$, whence $P$ is $l$-reflexive.
\end{proof}

Bey, Henk \& Wills~\cite[Prop.~1.8]{beyhenkwills2007notes} showed that every lattice polytope $P\in\P^n$ whose Ehrhart polynomial has only roots with real part
equal to $-\frac12$ is unimodularly equivalent to a reflexive polytope.
Their arguments apply in order to extend~\cite[Prop.~16]{kasprzyknill2012reflexive} as follows:

\begin{corollary}
Let $P\in\P^n$ be a lattice polytope with index $l$ that contains the origin in its interior.
If all roots of the Ehrhart polynomial of $P$ have real part equal to $-\frac1{2l}$, then, up to a unimodular transformation, $P$ is an $l$-reflexive polytope.
\end{corollary}
\begin{proof}
Recall that $-\gamma_1(P),\ldots,-\gamma_n(P)$ are the roots of the Ehrhart polynomial of $P$.
It is not hard to see that $\frac{\lE_{n-1}(P)}{\vol(P)}=\sum_{i=1}^n\gamma_i(P)$.
Since the roots come in conjugate pairs, we get $\lE_{n-1}(P)=\frac{n}{2l}\vol(P)$ and \autoref{prop_charact_l-refl}~\romannumeral3) shows that $P$ is $l$-reflexive.
\end{proof}

We close with an open question.

\begin{problem}
Do there exist lattice polytopes whose Ehrhart polynomial has only roots with real part equal to $-\frac1{a}$, for some $a\notin 2\N$?
\end{problem}

\bibliographystyle{amsalpha}
\bibliography{mybib}

\end{document}